\documentclass[12pt]{article}

\usepackage{verbatim}
\usepackage{amsmath}
\usepackage{amsthm}
\usepackage{eucal}
\usepackage{amssymb}
\usepackage{epsfig}
\newtheorem{theo}{Theorem}
\newtheorem{prop}[theo]{Proposition}

\newtheorem{lemma}[theo]{Lemma}
\newcommand {\pare}[1] {\left( {#1} \right)}
\newcommand {\cro}[1] {\left[ {#1} \right]}
\newcommand {\acc}[1] {\left\{ {#1} \right\}}
\newcommand {\nor}[1] { \left\| {#1} \right\|}
 
\def \R  {\mathbb{R}} 
\def \T  {\mathbb{T}} 
\def \Z  {\mathbb{Z}}

\def \ind {\hbox{ 1\hskip -3pt I}}
\newcommand {\va}[1] {\left| {#1} \right|}

\newtheorem{remarque}[theo]{Remarque}

\begin{document}
\title{\textsc{Large deviations for self-intersection local times in subcritical dimensions}}
\author{Cl\'{e}ment Laurent}
\maketitle
\begin{abstract}
Let $(X_t,t\geq 0)$ be a random walk on $\mathbb{Z}^d$. 
Let   $ l_t(x)=  \int_0^t \delta_x(X_s)ds$ be the local time at site $x$ and
$ I_t= \sum\limits_{x\in\mathbb{Z}^d} l_t(x)^p $ the p-fold self-intersection local time (SILT). Becker and K\"onig \cite{BK}  have recently proved a large deviations principle for $I_t$ for all $(p,d)\in\mathbb{R}^d\times\mathbb{Z}^d$ such that $p(d-2)<2$. We extend these results to a broader scale of deviations and to the whole subcritical domain $p(d-2)<d$. Moreover we unify the proofs of the large deviations principle using a method introduced by Castell \cite{Castell} for the critical case $p(d-2)=d$ and developed by Laurent \cite{Laurent} for the critical and supercritical case $p(d-\alpha)\geq d$ of $\alpha$-stable random walk.

\end{abstract}

\section{Introduction}
Let $(X_t,t\geq 0)$ be a continuous time simple random walk on $\mathbb{Z}^d$, started from the origin. We denote by $\Delta$ his generator given by  
$$\Delta f(x)= \sum\limits_{y\sim x} (f(y)-f(x))$$
where we sum over the nearest neighbors of $x$. 

Let $P$ and $E$ be the associated probability measure and expectation of the walk. In this article we are interested in the self-intersection local times (SILT):
$$ \forall p>1,\forall t>0, I_t=\sum_{x\in\mathbb{Z}^d}l_t(x)^p,\text{ where } \forall x\in\mathbb{Z}^d, \forall t>0, l_t(x)=\int_0^t \delta_x(X_s)ds.$$

The SILT measures how much the random walk does intersect itself. Indeed, it is easy to see that if $I_n$ is the discrete analogous of $I_t$, then we have $$n\leq I_n\leq n^p,$$ because if the random walk stays one unit of time in each visited site then $I_n=n$, and if the random walk stays all the time in one site, $I_n=n^p$. This is a first clue to understand the SILT. More precisely, we have some law of large numbers type results. In dimension $1$, we know that the random walk is  recurrent  positive, so the random walk does intersect itself a lot and $I_t\sim t^{(p+1)/2}$. In dimension $2$, the random walk is still recurrent but has more space to live and now $I_t\sim t(\log t)^{p-1}$ (see \cite{Cerny}). For higher dimension $d\geq 3$, the random walk is now transient and spends about one unit of time in each visited site, and $I_t\sim t$ (see \cite{BK2}). 
It is clear that the SILT is related with the intersections of the random walk when $p$ is an integer, because we can then rewrite the SILT as follows:

 $$I_t= \int_0^t\cdot\cdot\cdot \int_0^t \ind_{X_{t_1}=\cdot\cdot\cdot =X_{t_p}}dt_1\cdot\cdot\cdot dt_p.$$

Our ability to determine if the trajectory of the random walk is unfolded or very concentrated in a few number of sites, interests physicists and  notably statistical mechanicians (see for instance \cite{Bolthausen},\cite{West}). For instance, a polymer $(X_n,n\in[0,N])$ can be modeled as the trajectory of a nearest neighborhood random walk under Gibbs measure with Hamiltonian $\beta I_N$. This measure favors polymers with few intersections for $\beta>0$ whereas  it favors polymers with a lot of intersections when $\beta <0$.

The SILT is also relevant when we are interested in the random walk in random scenery (see for instance \cite{KS},\cite{AC},\cite{GKS},\cite{Asselah},\cite{FMW},\cite{Asselah5}). The model is the following. Let us consider a random walk $(X_t,t\geq 0)$ on $\mathbb{Z}^d$. The random scenery is an independent and identically distributed field $(Y_z,z\in\mathbb{Z}^d)$ independent of the walk. The associated random walk in random scenery is $Z_t=\int_0^t Y(X_s)ds$. It is easy to see that $Z_t= \sum\limits_{z\in\mathbb{Z}^d} Y(z)l_t(z)$ where $l_t(z)$ is the local time of the random walk $(X_t,t\geq 0)$. Now if we assume that $(Y(z),z\in\mathbb{Z}^d)$ is for instance a centered Gaussian field, then conditionally on $(X_s,0\leq s\leq t)$, $Z_t\sim \mathcal{N}(0,I_t)$.

We have seen some Law of large numbers results, let us state Central limit theorem in the most studied case $p=2$:

\begin{enumerate}
\item in dimension $d=1$, $\frac{I_t}{t^{3/2}}\xrightarrow {(d)} \gamma_1$, where $\gamma_1$ is the intersection local time of a Brownian motion (see \cite{Borodin},\cite{ChenLi},\cite{Perkins}).
\item in  dimension $d=2$, $\frac{I_t-E[I_t]}{t}\xrightarrow {(d)}\gamma_1'$, where $\gamma_1'$ is the renormalized intersection local time of a Brownian motion (see \cite{Dynkin},\cite{LeGall},\cite{Rosen},\cite{Stoll}), and  $E[I_t]\sim Ct\log(t)$ where C is an explicite constant (see \cite{Cerny}).
\item in dimension $d\geq 3$,  $\frac{I_t-E[I_t]}{\sqrt{\text{Var}(I_t)}}\xrightarrow {(d)} \mathcal{N}(0,1)$ (see \cite{Chen2}), where $E[I_t]\sim Ct$ with C an explicite constant (see \cite{BK2}), and $Var(I_t)\sim 
\left\{ \begin{array}{ll}
	t\log(t) \text{ if } d=3\, 
	\\
	  t \text{ if } d>4\, .
	\end{array} 
\right.$ 
(see \cite{BS}).
\end{enumerate}

Since the law of large numbers and limit laws have been established, it is natural to be interested in the large deviations of the SILT. The large deviations are the study of rare events. In this article we  wonder how $I_t$ can exceed its mean, i.e. we estimate the probability $P(I_t\geq t^pr_t^p)$ where $t^tr_t^p\gg E[I_t]$. Heuristically, it is interesting to ask how the walk can realize this kind of atypical event. We propose here a classical strategy for the walk to realize large deviations of its SILT.

Let us localize the walk in a ball of radius R up to time $\tau$.
On one hand, the walk arrives at the edge of the ball in $R^2$ units of time, and the probability of this localization is about $\exp (-\frac{\tau}{R^2})$. 
On the other hand, the walk spends  about $\frac{\tau}{R^d}$ units of time on each site of the ball, so $I_t$ increases to $\pare{\frac{\tau}{R^d}}^p R^d=\tau^p R^{d(1-p)}$.
We want $I_t=t^pr_t^p$, which gives $\tau=tr_t R^\frac{d(p-1)}{p}$. Thus the probability of this localization is about $\exp\pare{-tr_tR^{\frac{d(p-1)}{p}-2}}$. Maximizing this quantity in $R$, we obtain three cases:

\begin{enumerate}
\item$\frac{d(p-1)}{p}-2>0\Leftrightarrow p(d-2)>d$ (supercritical case): in this case the optimal choice for $R$ is $1$. A good strategy to realize the large deviations is to spend a time of order $tr_t$ in a ball of radius $1$, and then: $P(I_t\geq t^tr_t^p)\sim \exp(-tr_t)$.
\item$\frac{d(p-1)}{p}-2=0\Leftrightarrow p(d-2)=d$ (critical case): here the choice of $R$ does not matter. Every strategy consisting in spending a time of order $tr_t R^\frac{d(p-1)}{p}$ in a ball of radius $R$ such that $1\leq R\ll 1/\sqrt{r_t}$ could be a good strategy, so $P(I_t\geq t^pr_t^p)\sim \exp(-tr_t)$.
\item$\frac{d(p-1)}{p}-2<0\Leftrightarrow p(d-2)<d$ (subcritical case): a good strategy is to stay up to time $t$ in a ball of maximal radius, i.e. $\pare{\frac{1}{r_t}}^\frac{p}{d(p-1)}$, thus $P(I_t\geq tr_t)\sim\exp\pare{-tr_t^\frac{2p}{d(p-1)}}$. 
\end{enumerate}

The question of the large deviations for the SILT has been studied a lot during the last decade. We make a brief review of the results obtained so far.

In the subcritical case $p(d-2)<d$, the authors of \cite{ChenLi} and \cite{BCR2} use the limit object in the central limit theorem to solve the question. The large deviations principle is obtained for all $p\in]1,+\infty[$  in dimension $1$, and only for $p=2$ in dimension $2$ and $3$. The exact asymptotic of $\frac{1}{tr_t^{2p/d(p-1)}}\log P(I_t\geq tr_t)$ is given in terms of a variational formula, involving functions defined on $\mathbb{R}^d$ and related to Gagliardo-Nirenberg inequality. In a very recent paper, Becker and K\"onig \cite{BK} answer partly the question of the generalization to real value of $p$. Indeed, they prove a large deviations principle with two restrictions. The first one is a parameter-dimension restriction. They need $p(d-2)<2$ instead of $p(d-2)<d$. The second one is a scale restriction. The large deviations principle is only obtained for $I_t\gg (\log t)^\frac{d(p-1)}{d+2}t^\frac{2p+d}{d+2}$ which is larger than the scale of the mean of $I_t$ given previously. 

In the critical case $p(d-2)=d$, Castell \cite{Castell} uses  Eisenbaum isomorphism theorem (theorem \ref{Eisenbaum}) which links the law of the local times with the law of a Gaussian process.  The constant of large deviations is expressed in term of the best constant in a Sobolev inequality. 

In the supercritical case $p(d-2)>d$, Chen and M\"orters \cite{ChenMorters} proved a large deviations principle for integer value of $p$ computing large moments of the SILT. Asselah \cite{Asselah5} improved this result getting the large deviations principle up to the scale of the mean but only for $p=2$. Finally Laurent \cite{Laurent} proved a large deviations principle for all $p>1$ and for an $\alpha$-stable random walk with $p(d-\alpha)\geq d$ using the Eisenbaum isomorphism theorem. 

 A recent monograph of Chen \cite{Chen} summarizes these results, we refer to it for an exhaustive treatment of the subject.

In this paper, we extend the results of Becker and K\"onig \cite{BK} in two different directions. We obtain a large deviations principle down to the scale of the mean, and we take off their condition $p(d-2)<2.$ Furthermore we unify the proofs of the large deviations principle in the three different cases, proving that the method based on the Eisenbaum isomorphism theorem is also working in the subcritical case. 
\paragraph{Main results.}$ $\\

Let us introduce some notations to state our results.
We denote by $\nabla$ and by $\nor{\cdot}_p$ the continuous gradient and the $L^p$-norm of functions defined on $\mathbb{R}^d$, and by $q$ the conjugate of $p$.
\begin{theo}
\label{theo}
  Let $\chi_{d,p}:=\inf\acc{\frac{1}{2}\nor{\nabla g}_2^2, g\in L^2(\mathbb{R}^d)\cap L^{2p}(\mathbb{R}^d) \text{ such that } \nor{g}_2=\nor{g}_{2p}=1}$.
Assume that $p(d-2)<d$ and that \\
- in dimension $d=1$,  $\frac{1}{t^{1/2q}}\ll r_t\ll 1$\\
- in dimension $d=2$,  $\pare{\frac{\log t}{t}}^{1/q}\ll r_t\ll 1$\\
 - in dimension $d\geq 3$,  $\frac{1}{t^{1/q}}\ll r_t\ll 1$\\
 then we have
$$\lim_{t\rightarrow +\infty} \frac{1}{tr_t^{2q/d}} \log P\pare{I_t\geq  t^pr_t^p} = -\chi_{d,p}.$$
\end{theo}

\begin{remarque} About the scale of the deviations.\\
Note that our conditions on $r_t$ are equivalent to $t^pr_t^p\gg E[I_t]$.
We haven't succeeded to go under the scale of the mean. This question is still open with the exception of the dimension 2 and 3 for $p=2$ in the subcritical case $p(d-2)<d$ (see the monograph of Chen \cite{Chen}). In the case where the dimension is larger than 5 and for $p=2$ (supercritical case $p(d-2)>d$), Asselah \cite{Asselah5} has succeeded to obtain the constant of deviations up to the scale of the mean.
\end{remarque}

\begin{remarque} 
\label{nondegenere}
$\chi_{d,p}$ is non degenerate.\\
We prove that the constant $\chi_{d,p}$ is non degenerate linking it to the best constant in the  Gagliardo-Nirenberg inequality. We recall that the Gagliardo-Niremberg constant $K_{d,p}$  is defined by $$K_{d,p}=\sup_g\acc{\frac{\nor{g}_{2p}}{\nor{\nabla g}_2^{d/2q}\nor{g}_2^{1-d/2q}}}$$ 
and is a non degenerate constant in the subcritical case $p(d-2)<d$.
This expression being invariant under the transformation $g_\beta(\cdot)=\beta^{d/2p}g(\beta\cdot)$, we can take the supremum over $\nor{g}_2=1$. So,
$$K_{d,p}=\sup\acc{\frac{\nor{g}_{2p}}{\nor{\nabla g}_2^{d/2q}},\nor{g}_2=1}.$$
Again, we remark that this expression is invariant under the transformation $g_\beta(\cdot)=\beta^{d/2}g(\beta\cdot)$. So we can take the supremum over $\nor{g}_{2p}=1$ then $$K_{d,p}=\sup\acc{\nor{\nabla g}_2^{-d/2q},\nor{g}_2=\nor{g}_{2p}=1}.$$  
So $\chi_{d,p}=\frac{1}{2}K_{d,p}^{-4q/d}$.
\end{remarque}$ $\\
\textbf{Sketch of proof.}\\

The proof of the lower bound of the large deviations principle (Section 3) is quite classical. Let $L_t = \frac{\alpha_t^d}{t} l_t(\lfloor \alpha_t x\rfloor)$ be a rescaled version of $l_t$.
Gantert, K\"onig and Shi (lemma 3.1in \cite{GKS}) proved that for $R>0$, under the sub-probability measure $P(\cdot , supp(L_t)\subset [-R,R]^d)$, $L_t$ satisfies a large deviations principle on \\$\mathcal{F}=\acc{\psi\text{ such that }\nor{\psi}_2=1\text{ and } supp(\psi)\subset [-R,R]^d}$ with rate function $\mathcal{I}(\psi)=\frac{1}{2}\nor{\nabla \psi}_2^2$. Let $\nor{\cdot}_{p,R}$ be the $L^p$-norm of functions defined on $[-R,R]^d$, then the function 
$$\psi\in\mathcal{F}\rightarrow \nor{\psi}_{p,R}=\sup\acc{\int\limits_{[-R,R]^d}\psi(x)\phi(x),\nor{\phi}_{q,R}=1}$$
being lower semi-continuous, we can apply a contraction principle. 

For the upper bound, we can not proceed in the same way as $I_t$ is only a lower semi-continuous function of $l_t$.   Many different methods were developed in the papers written on the subject to overcome this difficulty. In this paper, we use the same method as Castell \cite{Castell} and Laurent \cite{Laurent}, i.e. the Eisenbaum isomorphism theorem (theorem \ref{Eisenbaum}).


Let us describe the proof of the upper bound in theorem \ref{theo} (Section 2). In step 1, we compare the SILT of the random walk with the SILT of the random walk projected on the discrete torus $\T_{R\alpha_t}$ of radius $R\alpha_t$, and stopped at an exponential time of parameter $\lambda_t$ independent of the walk (lemma \ref{tpsexp}). Then we apply in step 2 and 3 Eisenbaum's theorem (theorem \ref{Eisenbaum}) to arrive at a centered Gaussian process $(Z_x,x\in\T_{R\alpha_t})$ whose covariance is given by $G_{R\alpha_t,\lambda}(x,y)=E_x\cro{\int_0^\tau \delta_x(X_s^{R\alpha_t})ds}$ (lemma \ref{gauss}) where $(X_s^{R\alpha_t},s\geq 0)$ is the random walk projected on $\T_{R\alpha_t}$. In step 4 we work on the Gaussian process (lemma \ref{Z}) to obtain an upper bound $\rho_1(a,R,t)$ given by a discrete-space variational formula: 
$$\rho_1(a,R,t)=\inf\acc{\lambda_t N_{2,R\alpha_t}^2(h)+\frac{1}{2}N_{2,R\alpha_t}^2(\tilde\nabla h),h\in L^{2p}(\T_{R\alpha_t})\text{ such that }N_{2p,R\alpha_t}(h)=1},$$
where  $N_p(\cdot)$ is the $l^p$-norm of functions defined on $\mathbb{Z}^d$,  $N_{p,A}(\cdot)$ the $l^p$-norm of A-periodic functions defined on $\mathbb{Z}^d$, and $\tilde\nabla$ the discrete gradient defined by
$$\forall x \in\mathbb{Z}^d, \forall i\in\acc{1,...,d}, \tilde\nabla_i f(x)=f(x+e_i)-f(x),$$
where $(e_1,...,e_d)$ is the canonical base of $\mathbb{R}^d$. 
 Then step 5 is devoted to take the limit using the following proposition.
\begin{prop}
\label{limcste}
Let $\rho(a)=\inf\acc{a\nor{h}_2^2+\frac{1}{2}\nor{\nabla h}_2^2,h\in L^{2p}(\mathbb{R}^d)\text{ such that }\nor{h}_{2p}=1}$. \\
Assume that $\alpha_t=r_t^{-q/d}$ and $\lambda_t= a\alpha_t^{-2}=ar_t^{2q/d}$. If $p(d-2)<d$, then
  $$\liminf_{R\rightarrow +\infty}\liminf_{t\rightarrow + \infty}r_t^{1-2q/d}\rho_1(a,R,t)\geq \rho(a).$$
\end{prop}

The proof of proposition \ref{limcste} is inspired by the proof of Lemma 2.1 of Becker and K\"onig \cite{BK}. The main difficulty is to pass from the discrete-space variational formula giving $\rho_1(a,R,t)$ to the continuous-space variational formula giving $\rho(a)$. First we take a sequence $h_n$ of functions defined on $\mathbb{Z}^d$ that approach the infimum in the definition of $\rho_1(a,R,t)$. Then we extend these functions on $\mathbb{R}^d$ to build a sequence $g_n$ of continuous functions  defined on $\mathbb{R}^d$. This sequence $g_n$ of functions is our candidate to realize the infimum in the definition of $\rho(a)$.  Furthermore, by definition of $\rho_1(a,R,t)$ and $\rho(a)$, we want to control $N_{2,R\alpha_t}(h_n)$ by $\nor{g_n}_2$ and $N_{2,R\alpha_t}(\tilde\nabla h_n)$ by $\nor{\nabla g_n}_{2}$.  This control, combined with the continuity of the sequence $g_n$ is the main difficulty of the proof. 

We finally prove that the upper and the lower bound are equals with the following proposition:

 \begin{prop}
 \label{cste}
 Let $\rho(a)$ be as in theorem \ref{theo} and $\chi_{d,p}$ be as in proposition \ref{limcste}, then
 $$\inf\acc{a-\rho(a),a>0} =- \chi_{d,p}.$$
 \end{prop}

\section{Proof of the upper bound of Theorem \ref{theo} }
Let us begin with a lemma.
We denote by $p_s^{R\alpha_t}(\cdot,\cdot)$ the probability transition of the random walk $(X_s^{R\alpha_t},s\geq 0)$.

\begin{lemma}{Behavior of $G_{R\alpha_t,\lambda_t}(0,0)$}.\\
\label{G}
Assume that $\lambda_t=a\alpha_t^{-2}$ and $\lambda_t\rightarrow 0$. Then for any $a,R>0$,
\begin{enumerate}
\item for $d=1$, $G_{R\alpha_t,\lambda_t}(0,0)=O(\alpha_t)$.
\item for $d=2$, $G_{R\alpha_t,\lambda_t}(0,0)=O(\log\alpha_t)$.
\item for $d\geq 3$, $G_{R\alpha_t,\lambda_t}(0,0)=O(1)$.
\end{enumerate}
\end{lemma}

\begin{proof}
Applying theorems 3.3.15 and 2.3.1 in \cite{Saloff}, we know by Nash inequalities that $$\exists C>0 \text{ such that }\forall s>0,\ \left\vert p_s^{R\alpha_t}(0,0) -\frac{1}{(R\alpha_t)^d} \right\vert \leq \frac{C}{s^{d/2}}.$$ 
So
\begin{align*}
G_{R\alpha_t,\lambda_t}(0,0)&=\int_0^{+\infty} \exp(-s\lambda_t ) p_s^{R\alpha_t}(0,0)ds\\
&\leq 1 +\int_1^{+\infty} \exp(-s\lambda_t)\pare{\frac{1}{(R\alpha_t)^d}+\frac{C}{s^{d/2}}}ds\\
&\leq  1+\frac{1}{\lambda_t(R\alpha_t)^{d}}+\lambda_t^{d/2-1}\int_{\lambda_t}^{+\infty} \frac{C\exp(-u)}{u^{d/2}}du.
\end{align*}
As $\lambda_t\rightarrow 0$, for $t$ large enough,
\begin{align*}
G_{R\alpha_t,\lambda_t}(0,0)&\leq 1+\frac{1}{\lambda_t(R\alpha_t)^d} +C\lambda_t^{d/2-1}\pare{\int_{\lambda_t}^1 \frac{du}{u^{d/2}} +1}.
\end{align*}
Remember that $\lambda_t=a\alpha_t^{-2}$, then we have the result in the three cases.
\end{proof}

\subsection{Step 1: comparison with the SILT of the random walk on the torus stopped at an exponential
time}

\begin{lemma}
\label{tpsexp}
Let $\tau$ be an exponential time of parameter $\lambda_t=a\alpha_t^{-2}$.
Let $l_{R\alpha_t,\tau}(x)=\int_0^{\tau} 
\delta_x (X_s^{R\alpha_t}) \, ds$.
Then $\forall a,R,\alpha_t >0$: 

\begin{equation*}
P \cro{ N_p(l_t)\geq t r_t}
\leq e^{t\lambda_t} P \cro{ N_{p,R\alpha_t}(l_{R\alpha_t,\tau})\geq tr_t}. 
\end{equation*}
 \end{lemma}

\begin{proof}
We deduce by convexity that
\begin{align*}
N_p^p(l_t) &= \sum_{x \in \Z^d} l_t^p(x) 
	 = \sum_{x \in \T_{R\alpha_t}} \sum_{k \in \Z^d} l^p_{t}(x+kR\alpha_t)\\
    & \leq  \sum_{x \in \T_{R\alpha_t}} \pare{\sum_{k \in \Z^d} l_t(x+kR\alpha_t)}^p
      =\sum_{x \in \T_{R\alpha_t}} l^p_{R\alpha_t,t}(x) = N_{p,R\alpha_t}^p(l_{R\alpha_t,t}).
\end{align*}

Then using the fact that $\tau\sim\mathcal{E} (\lambda_t)$ is independent of $(X_s,s\geq 0)$ we get:
\begin{align*}
P \cro{ N_p(l_t)\geq tr_t} \exp\pare{-t\lambda_t}
		&\leq P \cro{N_{p,R\alpha_t}(l_{R\alpha_t,t})\geq tr_t} P(\tau\geq t)\\
		&= \mathbb{P} \cro{ N_{p,R\alpha_t}(l_{R\alpha_t,t})\geq tr_t,\tau\geq t}\\
		&\leq P \cro{ N_{p,R\alpha_t}(l_{R\alpha_t,\tau})\geq tr_t}.
\end{align*}
Finally, $P \cro{ N_p(l_t)\geq tr_t}
\leq e^{t\lambda_t} P \cro{ N_{p,R\alpha_t}(l_{R\alpha_t,\tau})\geq tr_t} $.

\end{proof}

\subsection{Step 2: the Eisenbaum isomorphism theorem}
\begin{theo} (Eisenbaum, see for instance corollary 8.1.2 page 364 in \cite{MarcusRosen}). 
\label{Eisenbaum}
Let $\tau$ be as in lemma \ref{tpsexp}  and
let $(Z_x, x \in \T_{R\alpha_t})$ be a centered Gaussian process 
with covariance matrix $G_{R\alpha_t,\lambda_t}=E_x\cro{\int_0^\tau \delta_x(X_s^{R\alpha_t})ds}$ 
independent of $\tau$  and of the  random walk $(X_s, s \geq 0)$. 
For $s\neq 0$, consider the process 
$S_x := l_{R\alpha_t,\tau}(x) + \frac{1}{2} (Z_x+s)^2$, then for all measurable
and bounded function $F : \R^{\T_{R\alpha_t}} \mapsto \R$: 
\begin{equation*}
E\cro{F((S_x; x\in \T_{R\alpha_t}))}
= E\cro{F\pare{(\frac{1}{2}(Z_x +s)^2;x \in \T_{R\alpha_t})} \, \pare{1 + \frac{Z_0}{s}}}
\, . 
\end{equation*} 
\end{theo}

\subsection{Step 3: Comparison between $N_{p,R\alpha_t}(l_{R\alpha_t,\tau})$ and $N_{2p,R\alpha_t}(Z)$}
\begin{lemma}
\label{gauss}
Let $\tau$ and $(Z_x, x\in \T_{R\alpha_t})$ be defined
as in theorem \ref{Eisenbaum}. 
 For all $\epsilon >0$, there exists a constant $C(\epsilon)$ such that $\forall  a,R,\alpha_t,r_t>0$:
\begin{align*} 
P\pare{N_{p,R\alpha_t}(l_{R\alpha_t,\tau})\geq tr_t } 
\leq &C(\epsilon)\pare{1+\frac{(R\alpha_t)^{d/2p}}{\epsilon\sqrt{2\epsilon tr_t\lambda_t}}}\frac{P\pare{N_{2p,R\alpha_t}(Z)\geq \sqrt{2tr_t}(1+\circ(\epsilon))}^{1/(1+\epsilon)}}{P(N_{2p,R\alpha_t}(Z)\geq (1+\epsilon)\sqrt{2tr_t\epsilon})}.  
\end{align*}  
\end{lemma}

\begin{proof}
\begin{align*}
 S_x :=l_{R\alpha_t,\tau}(x) + \frac{1}{2} (Z_x+s)^2 &
\Rightarrow  S_x^p\geq l^p_{R\alpha_t,\tau}(x) + \pare{\frac{1}{2} (Z_x+s)^2}^p\\
&\Rightarrow  N_{p,R\alpha_t}^p (S) \geq N_{p,R\alpha_t}^p(l_{R\alpha_t,\tau}) +\frac{1}{2^p}N_{2p,R\alpha_t}^{2p}(Z+s).
\end{align*}
By independence of $(Z_x,x\in\T_{R\alpha_t})$ with the random walk $(X_s,s\geq 0)$ and the exponential time $\tau$, we have $\forall \epsilon >0$,
\begin{align}
 \nonumber & P\pare{N_{p,R\alpha_t}^p(l_{R\alpha_t,\tau})\geq t^p r_t^p } P\pare{\frac{1}{2^p}N_{2p,R\alpha_t}^{2p}(Z+s)\geq t^p r_t^p\epsilon^p}\\
=  &\nonumber P\pare{N_{p,R\alpha_t}^p(l_{R\alpha_t,\tau})\geq t^pr_t^p, \frac{1}{2^p}N_{2p,R\alpha_t}^{2p}(Z+s)\geq t^pr_t^p\epsilon^p}\\
\leq & \nonumber P\pare{N_{p,R\alpha_t}^p(l_{R\alpha_t,\tau})+\frac{1}{2^p}N_{2p,R\alpha_t}^{2p}(Z+s)\geq t^pr_t^p(1+\epsilon^p)}\\
\leq &\nonumber P\pare{N_{p,R\alpha_t}^p(S)\geq t^pr_t^p(1+\epsilon^p)}\\
= &\label{PPE} E\cro{\pare{1+\frac{Z_0}{s}}; \frac{1}{2^p}N_{2p,R\alpha_t}^{2p}(Z+s)\geq t^pr_t^p(1+\epsilon^p)}
\end{align}
where the last equality comes from Theorem \ref{Eisenbaum}.
Moreover by H\"older's inequality, $\forall \epsilon >0$,

\begin{align}
\nonumber
&E\cro{\pare{1+\frac{Z_0}{s}} ; \frac{1}{2^p}N_{2p,R\alpha_t}^{2p}(Z+s)\geq t^pr_t^p(1+\epsilon^p)}\\
\label{EexpE}\leq &E\cro{\va{1+\frac{Z_0}{s}}^{1+1/\epsilon}}^{\epsilon/(1+\epsilon)}P\pare{N_{2p,R\alpha_t}^{2p}(Z+s)\geq 2^pt^pr_t^p(1+\epsilon^p)}^{1/(1+\epsilon)}.
\end{align}

Combining  (\ref{PPE}) and (\ref{EexpE}) we obtain that $\forall a,\epsilon >0$,
\begin{equation}
 \label{PexpEP}
 P\pare{N_{p,R\alpha_t}(l_{R\alpha_t,\tau})\geq tr_t } 
\leq E\cro{\va{1+\frac{Z_0}{s}}^{1+1/\epsilon}}^{\epsilon/(1+\epsilon)}
\frac{P\pare{N_{2p,R\alpha_t}(Z+s)\geq \sqrt{2tr_t}(1+\circ(\epsilon))}^{1/(1+\epsilon)}}{P(N_{2p,R\alpha_t}(Z+s)\geq \sqrt{2tr_t\epsilon})}  .
\end{equation}

Then using the fact that $Var(Z_0)=G_{R\alpha_t,\lambda_t}(0,0)\leq E[\tau]=\frac{1}{\lambda_t}$,
\begin{align}
\label{EEexp}
P\pare{N_{p,R\alpha_t}(l_{R\alpha_t,\tau})\geq tr_t } 
\leq &C(\epsilon)\pare{1+\frac{1}{s\sqrt{\lambda_t}}}\frac{P\pare{N_{2p,R\alpha_t}(Z+s)\geq \sqrt{2tr_t}(1+\circ(\epsilon))}^{1/(1+\epsilon)}}{P(N_{2p,R\alpha_t}(Z+s)\geq \sqrt{2tr_t\epsilon})}  .
\end{align}

Choosing $s=\frac{ \epsilon\sqrt{2tr_t\epsilon}}{(R\alpha_t)^\frac{d}{2p}} $, using triangle inequality and the fact that $N_{2p,R\alpha_t}(s)=s(R\alpha_t)^{\frac{d}{2p}}$,  we have:
\begin{align*} 
P\pare{N_{p,R\alpha_t}(l_{R\alpha_t,\tau})\geq tr_t } 
\leq &C(\epsilon)\pare{1+\frac{(R\alpha_t)^{d/2p}}{\epsilon\sqrt{2\epsilon tr_t\lambda_t}}}\frac{P\pare{N_{2p,R\alpha_t}(Z)\geq \sqrt{2tr_t}(1+\circ(\epsilon))}^{1/(1+\epsilon)}}{P(N_{2p,R\alpha_t}(Z)\geq (1+\epsilon)\sqrt{2tr_t\epsilon})}  
\end{align*}  

\end{proof}

\subsection{Step 4: Large deviations for $N_{2p,R}(Z)$}

\begin{lemma}
\label{Z}
Let $\tau$ and $(Z_x, x\in \T_{R\alpha_t})$ be defined
as in theorem \ref{Eisenbaum}, and
$\rho_1(a,R,t)$ be defined as in proposition \ref{limcste}.
Under assumptions of proposition \ref{limcste} and theorem \ref{theo},

\begin{enumerate}
\item $\forall a,R,t>0$, $\lambda_t\leq \rho_1(a,R,t)\leq aR^{d/q}\alpha_t^{d/q-2}$.
\item $\forall a,\epsilon,R,t>0$, 
\begin{equation*} 
P\cro{N_{2p,R\alpha_t} (Z)\geq \sqrt{tr_t \epsilon}}
\geq 
 \frac{1}{\sqrt{2 \pi tr_t \epsilon \rho_1(a,R,t)}} 
\pare{1-\frac{1}{tr_t \epsilon\rho_1(a,R,t)}} 
\exp\pare{- \frac{1}{2}tr_t \epsilon \rho_1(a,R, t)} .
\end{equation*} 
\item  $\forall a,\epsilon,R,t>0$,
\begin{align*} 
&P\pare{N_{2p,R\alpha_t}(Z)\geq \sqrt{2tr_t}(1+\circ(\epsilon)) }\\
\leq & \frac{\sqrt{2}}{(\sqrt{2tr_t}(1+\circ(\epsilon))+\circ(\sqrt{tr_t}))\sqrt{\pi\rho_1(a,R,t)}}\exp\pare{-\frac{\rho_1(a,R,t)\pare{\sqrt{2tr_t}(1+\circ(\epsilon))+\circ(\sqrt{tr_t})}^2}{2}}
.
\end{align*} 
\end{enumerate} 
\end{lemma} 

\begin{proof}
\begin{enumerate}
\item 
For the upper bound, it suffices to 
take $f=(R\alpha_t)^{-d/2p}$ to obtain the result.
For the lower bound, we remark that $N_{2p,R\alpha_t}(h)=1$ implies that for all $x\in\T_{R\alpha_t}, \ \vert h(x)\vert \leq 1$, and then $N_{2p,R\alpha_t}^{2p}(h)\leq N_{2,R\alpha_t}^2(h)$. Therefore $\rho_1(a,R,t)\geq \lambda_t$.
\item By H\"older's inequality, for any $f$ such that $\|f\|_{(2p)',R\alpha_t}=1$,
\[
P\cro{N_{2p,R\alpha_t} (Z)\geq \sqrt{tr_t \epsilon}}
 \geq P \cro{\sum_{x \in \T_{R\alpha_t}} f_x Z_x \geq \sqrt{ tr_t  \epsilon} }
\, .
\]
Since $\sum\limits_{x \in \T_{R\alpha_t}} f_x Z_x$ is a real centered Gaussian variable with
variance 
\[ \sigma^2_{a,R,t}(f)= 
\sum_{x,y  \in \T_{R\alpha_t}} G_{R\alpha_t,\lambda_t}(x,y) f_x f_y \, ,
\]
we have:
\begin{eqnarray*} 
P\cro{\nor{Z}_{2p,R\alpha_t} \geq  \sqrt{tr_t \epsilon}}
& \geq & 
\frac{\sigma_{a,R,t}(f)}{\sqrt{2\pi} \sqrt{tr_t \epsilon}} 
\pare{1 - \frac{\sigma^2_{a,R,t}(f)}{tr_t \epsilon}} 
\exp\pare{- \frac{tr_t \epsilon }{2 \sigma^2_{a,R,t}(f)}} 
 \\
& \geq & 
\frac{\sigma_{a,R,t}(f)}{\sqrt{2\pi} \sqrt{tr_t\epsilon}} 
\pare{1 - \frac{\rho_2(a,R,t)}{tr_t \epsilon}} 
\exp\pare{- \frac{tr_t\epsilon }{2 \sigma^2_{a,R,t}(f)}},
\end{eqnarray*}
where $\rho_2(a,R,t)=\sup\acc{\sigma^2_{a,R,t}(f), N_{(2p)',R\alpha_t}(f)=1}$.
Taking the supremum over $f$ we obtain that $\forall a,R,t,\epsilon >0$,
$$
P\cro{N_{2p,R\alpha_t} (Z)\geq \sqrt{tr_t \epsilon}}
\geq 
 \frac{\sqrt{\rho_2(a,R, t)}}{\sqrt{2 \pi tr_t \epsilon }} 
\pare{1-\frac{\rho_2(a,R,t)}{tr_t\epsilon}} 
\exp\pare{- \frac{tr_t \epsilon }{2 \rho_2(a,R, t)}}. 
$$ 
Then it suffices to prove that $\rho_2(a,R,t)=\frac{1}{\rho_1(a,R,t)}$ to have the result.

We denote by $<\cdot ,\cdot>_{R\alpha_t}$ the scalar product on $l^2(\T_{R\alpha_t})$. On one hand, by H\"older inequality, 
$$<f,G_{R\alpha_t,\lambda_t}f>_{R\alpha_t}\leq N_{2p,R\alpha_t}(G_{R\alpha_t,\lambda_t}f),\ \forall f \text{ such that }N_{(2p)',R\alpha_t}(f)=1.$$ 
Since $G_{R\alpha_t,\lambda_t}^{-1}=\lambda_t -\Delta$,
\begin{align*}
<f,G_{R\alpha_t,\lambda_t}f>_{R\alpha_t}&=<G_{R\alpha_t,\lambda_t}^{-1}G_{R\alpha_t,\lambda_t}f,G_{R\alpha_t,\lambda_t}f>_{R\alpha_t}\\
&= \lambda_tN_{2,R\alpha_t}^2(G_{R\alpha_t,\lambda_t}f)+\frac{1}{2}N_{2,R\alpha_t}^2(\nabla G_{R\alpha_t,\lambda_t} f)\\
&\geq \rho_1(a,R,t)N_{2p,R\alpha_t}^2(G_{R\alpha_t,\lambda_t}f).
\end{align*}
Therefore, for all $f$ such that $N_{(2p)',R\alpha_t}(f)=1$, $<f,G_{R\alpha_t,\lambda_t}f>_{R\alpha_t}^2\leq \frac{<f,G_{R\alpha_t,\lambda_t}f>_{R\alpha_t}}{\rho_1(a,R,t)}$. Then, taking the supremum over $f$, $\rho_2(a,R,t)\leq 1/\rho_1(a,R,t)$. 

On the other hand, let $f_0$ realizing the infimum in the definition of $\rho_1(a,R,t)$. 
\begin{align*}
\rho_2(a,R,t)&= \sup_{N_{(2p)',R\alpha_t}(f)=1} \acc{<f,G_{R\alpha_t,\lambda_t}f>_{R\alpha_t}}\\
&\geq \frac{<G_{R\alpha_t,\lambda_t}^{-1}f_0,f_0>_{R\alpha_t}}{N_{(2p)',R\alpha_t}^2(G_{R\alpha_t,\lambda_t}^{-1}f_0)}
= \frac{\rho_1(a,R,t)}{N_{(2p)',R\alpha_t}(G_{R\alpha_t,\lambda_t}^{-1}f_0)}.
\end{align*}
Furthermore, using the Lagrange multipliers method, we know that $N_{(2p)',R\alpha_t}^2(G_{R\alpha_t,\lambda_t}^{-1}h_0)=\rho_1(a,R,t)$. Hence $\rho_2(a,R,t)\geq 1/\rho_1(a,R,t)$, and then
$\rho_2(a,R,t)= 1/\rho_1(a,R,t)$. 
\item
Let $M$ be a median of $N_{2p,R\alpha_t}(Z)$. We can easily see that
\begin{equation}
\label{EEexp2}
P\pare{N_{2p,R\alpha_t}(Z)\geq \sqrt{2tr_t}(1+\circ(\epsilon)) }
\leq P\pare{\va{N_{2p,R\alpha_t}(Z)-M}\geq \sqrt{2tr_t}(1+\circ(\epsilon))-M}.
 \end{equation}

Using concentration inequalities for norms of Gaussian processes (see for instance lemma 3.1 in \cite{LT}), $\forall u > 0$, 

$P\cro{\va{N_{2p,R\alpha_t} (Z)- M} \geq \sqrt{u}} 
\leq 2 P( Y \geq \sqrt{\frac{u}{\rho_2(a, R, t)} }) $
where $Y\sim \mathcal{N}(0,1)$. Then for $tr_t\gg M^2$,
\begin{align}
\nonumber & P\pare{\va{N_{2p,R\alpha_t}(Z)-M}\geq \sqrt{2tr_t}(1+\circ(\epsilon))-M}\\
\nonumber \leq & 2P\pare{Y\geq \frac{\sqrt{2tr_t}(1+\circ(\epsilon))-M}{\sqrt{\rho_2(a,R,t)}}}\\
 \leq &\frac{2\sqrt{\rho_2(a,R,t)}}{(\sqrt{2tr_t}(1+\circ(\epsilon))-M)\sqrt{2\pi}}\exp\pare{-\frac{\pare{\sqrt{2tr_t}(1+\circ(\epsilon))-M}^2}{2\rho_2(a,R,t)}}.
\end{align}
Let us now prove that under our assumptions we have $tr_t\gg M^2$.
Since $M=(\text{median}(\sum\limits_{x\in\T_{R\alpha_t}} Z_x^{2p}))^{1/2p}$ and that for $X\geq 0,\ \text{ median}(X)\leq 2 E[X]$, we get:
\begin{align*}
M^2 &=(\text{median}(\sum_{x\in\T_{R\alpha_t}} Z_x^{2p}))^{1/p}\\
&\leq (2E[\sum_{x\in\T_{R\alpha_t}} Z_x^{2p}])^{1/p}\\
&\leq C(p)(\sum_{x\in\T_{R\alpha_t}} G_{R\alpha_t,\lambda_t}(0,0)^p E[Y^{2p}])^{1/p},\text{ where } Y\sim\mathcal{N}(0,1)\\
&\leq C(p)(R\alpha_t)^{d/p}G_{R\alpha_t,\lambda_t}(0,0)(E[Y^{2p}])^{1/p}\\
&\leq C(p)(R\alpha_t)^{d/p}G_{R\alpha_t,\lambda_t}(0,0).
\end{align*}
Recall that we have $\lambda_t=a\alpha_t^{-2}$ and $\alpha_t=r_t^{-q/d}$. \\
For $d=1$, by lemma \ref{G}, $M^2 = O (\alpha_t^{1+1/p})=O(r_t^{-\frac{p+1}{p-1}})$. Then
as  $r_t\gg \frac{1}{t^{1/2q}}$ we have $M^2\ll tr_t$.\\
For $d=2$, by lemma \ref{G}, $M^2=O(\alpha_t^{2/p}\log\alpha_t)=O\pare{r_t^{-1/(p-1)}\log\frac{1}{r_t}}$. Then as $r_t\gg\pare{\frac{\log t}{t}}^{1/q}$ we have $M^2\ll tr_t$.\\
For $d\geq 3$, by lemma \ref{G}, $M^2\leq C\alpha_t^{d/p}=Cr_t^{-1/(p-1)}$. Then as 
$r_t\gg t^{-1/q}$, we have $M^2\ll tr_t$.

\end{enumerate}
\end{proof}

\subsection{End of proof of the upper bound in theorem \ref{theo}}
Combining Lemma  \ref{tpsexp} and Lemma \ref{gauss} we have proved that: $\forall \epsilon,a,R,t>0$,
\begin{equation}
 P \cro{ N_p(l_t)\geq tr_t} 
\leq \label{PexpEP2}C(\epsilon)\exp(t\lambda_t)\pare{1+\frac{(R\alpha_t)^\frac{d}{2p}}{\epsilon \sqrt{2\epsilon tr_t \lambda_t}}} 
\frac{P\pare{N_{2p,R\alpha_t}(Z)\geq \sqrt{2tr_t}(1+\circ(\epsilon))}^{1/(1+\epsilon)}}
{P\cro{ N_{2p,R}(Z) \geq 2 \sqrt{2tr_t\epsilon}}}.
\end{equation}

First we look for an upper bound for the numerator in (\ref{PexpEP2}).
By 1 and 3 of lemma \ref{Z}, we have that
   \begin{equation}
   \label{limE}
   \limsup_t \frac{1}{tr_t^{2q/d}}\log P\pare{N_{2p,R\alpha_t}(Z)\geq \sqrt{2tr_t}(1+\circ(\epsilon))}^{1/(1+\epsilon)} \leq -\liminf\limits_{t\rightarrow +\infty} r_t^{1-2q/d}\rho_1(a,R,t)(1+\circ(\epsilon)).
   \end{equation}
  
 Now we work on the denominator in (\ref{PexpEP2}). Using 1 and 2 of lemma \ref{Z}, we obtain:
 \begin{align*}
   P\cro{ N_{2p,R}(Z) \geq 2 \sqrt{2tr_t\epsilon}}
&\geq 
 \frac{1}{\sqrt{16 \pi tr_t \epsilon \rho_1(a,R,t)}} 
\pare{1-\frac{1}{8tr_t \epsilon\rho_1(a,R,t)}} 
\exp\pare{- 4tr_t \epsilon \rho_1(a,R, t)} \\
& \geq \frac{1}{\sqrt{16 \pi ta \epsilon R^{d/q}r_t^{2q/d}}} 
\pare{1-\frac{1}{8tr_t\epsilon\lambda_t}} 
\exp\pare{- 4 \epsilon atr_t^{2q/d}R^{d/q}}
   \end{align*}
   Therefore,
\begin{equation}
\label{limP}
   \liminf_t \frac{1}{tr_t^{2q/d}}\log   P\cro{ N_{2p,R}(Z) \geq 2 \sqrt{2tr_t\epsilon}} \geq -4a\epsilon R^{d/q}.
   \end{equation}
     
  Now we  combine (\ref{PexpEP2}),(\ref{limE}),(\ref{limP}) to have:   
$$\limsup_t \frac{1}{tr_t^{2q/d}} \log P\pare{N_p(l_t)\geq tr_t} \leq a - (1+\circ(\epsilon))\liminf_t r_t^{1-2q/d}\rho_1(a,R,t)+4a\epsilon R^{d/q-2}.$$

Then we let $\epsilon\rightarrow 0$:
$$\limsup_t \frac{1}{tr_t^{2q/d}} \log P\pare{N_p(l_t)\geq tr_t} \leq a - \liminf_t r_t^{1-2q/d}\rho_1(a,R,t).$$
Then we take the limit over $R$ using proposition \ref{limcste}:
$$\limsup_t \frac{1}{tr_t^{2q/d}} \log P\pare{N_p(l_t)\geq tr_t} \leq a - \rho(a).$$

 We finish the proof  taking the infimum over $a>0$ and using proposition \ref{cste}.

\section{Proof of the lower bound of theorem \ref{theo}}

\begin{proof}

Let $\forall x\in\R^d$, $L_t = \frac{\alpha_t^d}{t} l_t(\lfloor \alpha_t x\rfloor)$ be the rescaled version of $l_t$. Thanks to the work of Gantert, K\"onig and Shi (lemma 3.1 in \cite{GKS}) we know that for $R>0$, under the sub-probability measure $P(\cdot , supp(L_t)\subset [-R,R]^d)$, $L_t$ satisfies a large deviations principle on \\$\mathcal{F}=\acc{\psi \ such\ that\ \nor{\psi}_2=1, supp(\psi)\subset [-R,R]^d}$ with rate function $\mathcal{I}(\psi)=\frac{1}{2}\nor{\nabla \psi}_2^2$ and speed $t\alpha_t^{-2}$.

So for $r_t=\alpha_t^{-d/q}$,
\begin{align*}
P(N_p(l_t)\geq tr_t)&=P(\nor{L_t}_p\geq 1)\\
&\geq P(\nor{L_t}_p>1,supp(L_t)\subset [-R,R]^d).
\end{align*}
Then, as $\nu\rightarrow\nor{\nu}_p=\sup\acc{\int\limits_{\R^d}f(x)d\nu(x),\nor{f}_q=1}$ is a lower semi-continuous function in $\tau$-topology, we have,
$$
\liminf_{t\rightarrow +\infty}\frac{1}{tr_t^{2q/d}}\log P(N_p(l_t)\geq tr_t)
\geq -\inf\acc{\frac{1}{2}\nor{\nabla\psi}_2^2,\nor{\psi}_2=1,\nor{\psi}_{2p}>1,supp(\psi)\subset [-R,R]^d}
.$$
Let $R\rightarrow +\infty$,
\begin{align*}
\liminf_{t\rightarrow +\infty}\frac{1}{tr_t^{2q/d}}\log P(N_p(l_t)\geq tr_t)
&\geq -\inf\acc{\frac{1}{2}\nor{\nabla\psi}_2^2,\nor{\psi}_2=1,\nor{\psi}_{2p}>1}\\
&= -\inf\acc{\frac{1}{2}\nor{\nabla\psi}_2^2,\nor{\psi}_2=\nor{\psi}_{2p}=1}.
\end{align*}

\end{proof}

\section{Proof of propositions \ref{limcste} and \ref{cste}}

We denote by $\mathfrak{S}_d$ the set of the permutations on $\acc{1,...,d}$, by  $\lfloor\cdot\rfloor$ the integer part, by $B(s)$ the ball of radius $s$ and by $Vol(B(s))$ its volume.\\
$ $\\
\textbf{Proof of proposition \ref{limcste}:}\\
Let choose a sequence $(R_n,t_n,h_n)$ such that $R_n\rightarrow +\infty$, $t_n\rightarrow +\infty$,$N_{2p,R\alpha_{t_n}}(h_n)=1$ and such that
\begin{align*}
 &\liminf_{R\rightarrow +\infty}\liminf_{t\rightarrow + \infty} \inf\acc{\frac{a}{\alpha_t^{d/q}}N_{2,R\alpha_t}^2(h)+\frac{1}{2}\alpha_t^{2-d/q}N_{2,R\alpha_t}^2(\tilde\nabla h),N_{2p,R\alpha_t}(h)=1}\\
 \geq &\frac{a}{\alpha_{t_n}^{d/q}}  N_{2,R_n\alpha_{t_n}}^2(h_n)+\frac{1}{2}\alpha_{t_n}^{2-d/q}N_{2,R_n\alpha_{t_n}}^2 (\tilde\nabla h_n)-\frac{1}{n}.
 \end{align*}
 $h_n$ is a sequence of functions defined on $\T_{R_n\alpha_{t_n}}$. We want to extend these functions to $\mathbb{R}^d$. In this perspective we split $\mathbb{Z}^d$ into cubes $C(k)=\times_{i=1}^d[k_i,k_i+1]$ for any $k\in\mathbb{Z}^d$. Then, each cube $C(k)$ is again splitted into $d!$ tetrahedra $T_\sigma(k)$, where for any $\sigma\in\mathfrak{S}_d$, $T_\sigma(k)$ is the convex hull of $k,k+e_{\sigma(1)},...,k+e_{\sigma(1)}+\cdot\cdot\cdot +e_{\sigma(d)}$. For any $y\in\mathbb{R}^d$ we denote by $\sigma(y)$ the unique permutation which defined the tetrahedra $T_{\sigma(y)}(\lfloor y\rfloor)$ where lives $y$. 
 Set  for any $x\in \R^d$,  
 \begin{equation}
 \label{defgn}g_n(x)=\alpha_{t_n}^{d/2p} h_n(\lfloor \alpha_{t_n} x\rfloor)+\alpha_{t_n}^{d/2p}\sum_{i=1}^d f_{n,\sigma(\alpha_{t_n} x),i} (\alpha_{t_n} x),
 \end{equation}
 where $\forall y\in\R^d$, $\forall i\in\acc{1,...,d}$, $\forall \sigma\in\mathfrak{S}^d$, 
 $$f_{n,\sigma,i}(y)=\pare{h_n(\lfloor y\rfloor+e_{\sigma(1)}+...+e_{\sigma(i)})-h_n(\lfloor y\rfloor+e_{\sigma(1)}+...+e_{\sigma(i-1)})} (y_{\sigma(i)}-\lfloor y_{\sigma(i)}\rfloor).$$
 
Following the work of Becker and K\"onig, it can be proved that $g_n$ is well-defined continuous and $R_n$-periodic. Now we set $\Psi_{R_n}$ a truncation function that verify 
$\Psi_{R_n}=\bigotimes\limits_{i=1}^{d} \psi_{R_n}:\R ^d\rightarrow [0,1]$, where $\psi_{R_n}=
\left\{ \begin{array}{ll}
	0 \text{ outside }[-R_n,R_n]\, 
	\\
	  \text{linear in } [-R_n,-R_n+R_n^\epsilon]\text{ and in }[R_n-R_n^\epsilon,R_n]\,
	  \\
	  1 \text{ in } [-R_n+R_n^\epsilon, R_n-R_n^\epsilon] .
	\end{array} 
\right.$ \\
The function $\frac{g_n \Psi_{R_n}}{\nor{g_n \Psi_{R_n}}_{2p}}$ is 
our candidate to realize the infimum in the definition of $\rho(a)$. Therefore we have to bound from below $N_2^2(\tilde\nabla h_n)$ by $\nor{\nabla (g_n\Psi_n)}_2^2$
and $N_2^2(h_n)$ by $\nor{ (g_n\Psi_n)}_2^2$. 
\\

Let us first work on the norm of the gradient.
Thanks to the definition (\ref{defgn}) of $g_n$, it is easy to see that  
\begin{equation}
\label{nggn}\nor{\nabla g_n}_{2,R_n}^2=\alpha_{t_n}^{2-d/q}N_{2,R_n\alpha_{t_n}}^2(\tilde\nabla h_n).
\end{equation}
 By the work of Becker and K\"onig, we know that 
 \begin{equation}
\label{mnggp}\nor{\nabla (g_n \Psi_{R_n})}_2^2 \leq \pare{1+\frac{1}{R_n^\epsilon}}\nor{\nabla g_n}_{2,R_n}^2+\frac{2}{R_n^\epsilon} \nor{ g_n}_{2,R_n}^2.
\end{equation}
Then we have to bound from above $\nor{g_n}_{2,R_n}^2$. Using another time the definition (\ref{defgn}) of $g_n$  and the triangle inequality, we have that 
\begin{equation}
\label{majnorgn}
\nor{ g_n}_{2,R_n}\leq \alpha_{t_n}^{-d/2q}N_{2,R_n\alpha_{t_n}}(h_n)+\alpha_{t_n}^{-d/2q}\nor{\sum_{i=1}^d f_{n,\sigma(\cdot),i}(\cdot)}_{2,R_n\alpha_{t_n}}.
\end{equation}

We bound from above now the norm of $ f_{n,\sigma(y),i}(y)$ for any $y\in T_\sigma(k)$: 
\begin{align*}
|\sum_{i=1}^d f_{n,\sigma(y),i}(y)|^2
 &\leq d \sum_{i=1}^d |h_n(\lfloor y\rfloor +e_{\sigma(1)}+\cdot\cdot\cdot +e_{\sigma(i)})-h_n(\lfloor y\rfloor +e_{\sigma(1)}+\cdot\cdot\cdot +e_{\sigma(i-1)})|^2.
\end{align*}
Then, 
\begin{align*}
\nor{\sum_{i=1}^d f_{n,\sigma,i}}_{2,R\alpha_{t_n}}^2
&\leq d\sum_{k\in B(R_n\alpha_{t_n})}\sum_{\sigma\in\mathfrak{S}(d)}\int\limits_{y\in T_{\sigma}(k)} \sum_{i=1}^d\vert \tilde\nabla_{\sigma(i)} h_n(k+e_{\sigma (1)}+\cdot\cdot\cdot+e_{\sigma(i-1)})\vert^2dy\\
&= \frac{d}{d!} \sum_{\sigma\in\mathfrak{S}(d)} \sum_{i=1}^d \sum_{k\in B(R_n\alpha_{t_n})}\vert\tilde\nabla_{\sigma(i)} h_n(k+e_{\sigma (1)}+\cdot\cdot\cdot+e_{\sigma(i-1)})\vert^2\\
&= \frac{d}{d!} \sum_{k\in B(R_n\alpha_{t_n})}\sum_{\sigma\in\mathfrak{S}(d)} \sum_{i=1}^d \vert\tilde\nabla_{i} h_n(k+e_1+\cdot\cdot\cdot+e_{i-1})\vert^2\\
&= d N_{2,R_n\alpha_{t_n}}^2(\tilde\nabla (h_n)).
\end{align*}

Using (\ref{nggn}):
\begin{equation}
\label{majnorfn}
\nor{\sum_{i=1}^d f_{n,\sigma(\cdot),i}(\cdot)}_{2,R_n\alpha_{t_n}}\leq \sqrt{d}N_{2,R_n\alpha_{t_n}}(\tilde\nabla h_n)=\frac{\sqrt{d}}{\alpha_{t_n}^{1-d/2q}}\nor{\nabla g_n}_{2,R_n} .
\end{equation}

Combining (\ref{majnorgn}) and (\ref{majnorfn}) we have: 
\begin{equation}
\label{mngn}\nor{g_n}_{2,R_n}\leq \alpha_{t_n}^{-d/2q} N_{2,R_n\alpha_{t_n}}(h_n)+ \sqrt{d}\alpha_{t_n}^{-1}\nor{\nabla g_n}_{2,R_n}.
\end{equation}

Putting together (\ref{nggn}),(\ref{mnggp}) and (\ref{mngn}), we have:
\begin{align*}
\nor{\nabla g_n}_{2,R_n}^2 &\geq \frac{R_n^\epsilon}{1+R_n^\epsilon}\nor{\nabla (g_n \Psi_{R_n})}_{2}^2
-\frac{2}{1+R_n^\epsilon}\pare{\frac{1}{\alpha_{t_n}^{d/2q}}N_{2,R_n\alpha_{t_n}}(h_n)+\frac{\sqrt{d}}{\alpha_{t_n}}\nor{\nabla g_n}_{2,R_n}}^2\\
&\geq \frac{R_n^\epsilon}{1+R_n^\epsilon}\nor{\nabla (g_n \Psi_{R_n})}_{2}^2
-\frac{2}{1+R_n^\epsilon}\left(\frac{1}{\alpha_{t_n}^{d/q}}N_{2,R_n\alpha_{t_n}}^2(h_n)+\frac{d}{\alpha_{t_n}^2}\nor{\nabla g_n}_{2,R_n}^2\right.\nonumber\\
&\left.  \hspace{7.3cm}+\frac{2\sqrt{d}}{\alpha_{t_n}^{1+d/2q}}N_{2,R_n\alpha_{t_n}}(h_n)\nor{\nabla g_n}_{2,R_n}\right)\\
&\geq \frac{R_n^\epsilon}{1+R_n^\epsilon}\nor{\nabla (g_n \Psi_{R_n})}_{2}^2
-\frac{2}{1+R_n^\epsilon}\left(\frac{1}{\alpha_{t_n}^{d/q}}N_{2,R_n\alpha_{t_n}}^2(h_n)+\frac{d}{\alpha_{t_n}^2}\nor{\nabla g_n}_{2}^2
\right.\\
&\left.\hspace{7.3cm}
+\frac{\sqrt{d}}{\alpha_{t_n}^{1+d/2q}}N_{2,R_n\alpha_{t_n}}^2  (h_n)+\frac{\sqrt{d}}{\alpha_{t_n}^{1+d/2q}}   \nor{\nabla g_n}_{2,R_n}^2\right).
\end{align*}

So, 
\begin{align*}
\nonumber&\nor{\nabla g_n}_{2,R_n}^2\pare{1+\frac{2}{1+R_n^\epsilon}\pare{\frac{d}{\alpha_{t_n}^2}+\frac{\sqrt{d}}{\alpha_{t_n}^{1+d/q}}}}\\
\geq &\frac{R_n^\epsilon}{1+R_n^\epsilon}\nor{\nabla (g_n \Psi_{R_n})}_2^2  - \frac{2}{1+R_n^\epsilon}\pare{\frac{1}{\alpha_{t_n}^{d/q}}+\frac{\sqrt{d}}{\alpha_{t_n}^{1+d/2q}}}N_{2,R_n\alpha_{t_n}}^2(h_n).
\end{align*}

Then, using (\ref{nggn}), the fact that $\alpha_{t_n}\rightarrow +\infty$ and that $1+d/2q> d/q$,  we obtain:
\begin{align}
\nonumber\alpha_{t_n}^{2-d/q} N_{2,R_n\alpha_{t_n}}^2(\tilde\nabla h_n)  \geq &\frac{R_n^\epsilon}{1+R_n^\epsilon+2\pare{\frac{d}{\alpha_{t_n}^2}+\frac{\sqrt{d}}{\alpha_{t_n}^{1+d/q}}}}
\nor{\nabla (g_n \Psi_{R_n})}_2^2   \\
\nonumber&-\frac{2}{1+R_n^\epsilon+2\pare{\frac{d^2}{\alpha_{t_n}^2}+\frac{\sqrt{d}}{\alpha_{t_n}^{1+d/q}}}}     \pare{\frac{1}{\alpha_{t_n}^{d/q}}+\frac{\sqrt{d}}{\alpha_{t_n}^{1+d/2q}}}N_{2,R_n\alpha_{t_n}}^2(h_n)\\
\label{mnghn}\geq &\pare{1- \frac{C}{R_n^\epsilon}}\nor{\nabla (g_n \Psi_{R_n})}_2^2 -\frac{C}{\alpha_{t_n}^{d/q}R_n^\epsilon}N_{2,R_n\alpha_{t_n}}^2(h_n).
\end{align}
Now we work on the $l_2$-norm of  $h_n$. Taking  the square in (\ref{mngn}), we have that $\forall \delta >0,$ 
$$\nor{g_n}_{2,R_n}^2 \leq (1+\frac{1}{\delta})\alpha_{t_n}^{-d/q} N_{2,R_n\alpha_{t_n}}^2(h_n)+ (1+\delta)d\alpha_{t_n}^{-2}\nor{\nabla g_n}_{2,R_n}^2.$$
So, 
\begin{align}
\nonumber\alpha_{t_n}^{-d/q}N_{2,R_n\alpha_{t_n}}^2(h_n)  &\geq \frac{\delta}{\delta +1}\nor{g_n}_{2,R_n}^2-\delta d\alpha_{t_n}^{-2}\nor{\nabla g_n}_{2,R_n}^2  \\
\label{mnhn}&\geq  \frac{\delta}{\delta +1}\nor{g_n\Psi_{R_n}}_{2}^2-\delta d\alpha_{t_n}^{-2}\nor{\nabla g_n}_{2,R_n}^2.
\end{align}

At this point of the proof, we have bounded from below $N_2^2(\tilde\nabla h_n)$ by $\nor{\nabla (g_n\Psi_{R_n})}_2^2$ (\ref{mnghn})
and $N_2^2(h_n)$ by $\nor{ (g_n\Psi_{R_n})}_2^2$ (\ref{mnhn}). Therefore, combining these two results:
\begin{align}
\nonumber& \frac{a}{\alpha_{t_n}^{d/q}}  N_{2,R_n\alpha_{t_n}}^2(h_n)+\frac{1}{2}\alpha_{t_n}^{2-d/q}N_{2,R_n\alpha_{t_n}}^2 (\tilde\nabla h_n)\\
 \nonumber\geq & a \pare{\frac{\delta}{\delta +1}\nor{g_n\Psi_{R_n}}_{2}^2-\delta d\alpha_{t_n}^{-2}\nor{\nabla g_n}_{2,R_n}^2} +\frac{1}{2}\pare{1- \frac{C}{R_n^\epsilon}}\nor{\nabla (g_n \Psi_{R_n})}_2^2 -\frac{C}{\alpha_{t_n}^{d/q}R_n^\epsilon}N_{2,R_n\alpha_{t_n}}^2(h_n)\\
 \nonumber\geq &\min\pare{\frac{\delta}{\delta +1},1- \frac{C}{R_n^\epsilon}} \pare{a\nor{g_n\Psi_{R_n}}_2^2+\frac{1}{2}\nor{\nabla (g_n \Psi_{R_n})}_{2}^2}
 -a\delta d\alpha_{t_n}^{-2}\nor{\nabla g_n}_{2,R_n}^2 -C\alpha_{t_n}^{-d/q}R_n^{-\epsilon}N_{2,R_n\alpha_{t_n}}^2(h_n)\\
\label{diff}\geq  &\min\pare{\frac{\delta}{\delta +1},1- \frac{C}{R_n^\epsilon}} \rho(a) \nor{g_n\Psi_{R_n}}_{2p}^2-a\delta d\alpha_{t_n}^{-2}\nor{\nabla g_n}_{2,R_n}^2 -C\alpha_{t_n}^{-d/q}R_n^{-\epsilon}N_{2,R_n\alpha_{t_n}}^2(h_n).
\end{align}

Now we show that we can assume that
 \begin{equation}
 \label{gna}
 \pare{1-CR_n^\frac{\epsilon-1}{2p}}\nor{g_n}_{2p,R_n}\leq \nor{g_n\Psi_{R_n}}_{2p}.
 \end{equation}
Indeed, for $a\in B_{R_n}$ let $g_{n,a}(x)=g_n(x-a)$. By periodicity of $g_n$, in one side we have 
\begin{align*}
\int\limits_{B_{R_n}}\int\limits_{B_{R_n}\backslash B_{R_n-R_n^\epsilon}}g_n^{2p}(x-a)dx\ da
&=\int\limits_{B_{R_n}\backslash B_{R_n-R_n^\epsilon}}\int\limits_{B_{R_n}}g_n^{2p}(x-a)dx\ da \\
&=\int\limits_{B_{R_n}\backslash B_{R_n-R_n^\epsilon}}\int\limits_{B_{R_n}}g_n^{2p}(x)dx\ da \\
 &\leq CR_n^{d-1+\epsilon }\nor{g_n}_{2p,R_n}^{2p},
\end{align*}
and on the opposite side we have
\begin{align*}
\int\limits_{B_{R_n}}\int\limits_{B_{R_n}\backslash B_{R_n-R_n^\epsilon}}g_n^{2p}(x-a)dx\ da
\geq Vol(B(1))R_n^d \inf\limits_{a\in B_{R_n}}\acc{\int\limits_{B_{R_n}\backslash B_{R_n-R_n^\epsilon}}g_n^{2p}(x-a)dx\ da}.
\end{align*}
Therefore
$$\inf\limits_{a\in B_{R_n}}\acc{\int\limits_{B_{R_n}\backslash B_{R_n^\epsilon}}g_n^{2p}(x-a)dx\ da}\leq C R_n^{\epsilon-1}\nor{g_n}_{2p,R_n}^{2p}.
$$
Remark that $g_n$ being periodic, for any $a\in B_{R_n}$, $\nor{g_n}_{2}=\nor{g_{n,a}}$ and $\nor{\nabla g_n}_2=\nor{\nabla g_{n,a}}_2$ . So we can assume (\ref{gna}). Hence, combining (\ref{diff}) and (\ref{gna}) we obtain:

\begin{align*}
&\frac{a}{\alpha_{t_n}^{d/q}}  N_{2,R_n\alpha_{t_n}}^2(h_n)+\frac{1}{2}\alpha_{t_n}^{2-d/q}N_{2,R_n\alpha_{t_n}}^2(\tilde\nabla h_n)\\
\geq &\min\pare{\frac{\delta}{\delta +1},1- \frac{C}{R_n^\epsilon}} \rho(a) \nor{g_n}_{2p,R_n}^2  \pare{1-CR_n^\frac{\epsilon-1}{2p}}^2-a\delta d\alpha_{t_n}^{-2}\nor{\nabla g_n}_{2,R_n}^2 -C\alpha_{t_n}^{-d/q}R_n^{-\epsilon}N_{2,R_n\alpha_{t_n}}^2(h_n).
\end{align*}
The problem is now to eliminate the norm of $g_n$. Using the definition (\ref{defgn}) of $g_n$ and the triangle inequality,
 we can prove that $$\nor{g_n}_{2p,R_n}\geq 1-C\alpha_{t_n}^{-1}\nor{\nabla g_n}_{2p,R_n}.$$
 Therefore, $\forall \gamma >0$,
$$ (1+\gamma)\nor{g_n}_{2p,R_n}^2+\frac{1+\gamma}{\gamma}C\alpha_{t_n}^{-2}\nor{\nabla g_n}_{2p,R_n}^2
\geq 1. $$
Finally,  
\begin{align*}
&\frac{a}{\alpha_{t_n}^{d/q}}  N_{2,R_n\alpha_{t_n}}^2(h_n)+\frac{1}{2}\alpha_{t_n}^{2-d/q}N_{2,R_n\alpha_{t_n}}^2(\tilde\nabla h_n)\\
\geq &\min\pare{\frac{\delta}{\delta+1},1- \frac{C}{R_n^\epsilon}}   \pare{1-CR_n^\frac{\epsilon-1}{2p}}^2\rho(a)\pare{\frac{1}{1+\gamma}  -
\frac{C}{\gamma\alpha_{t_n}^2}\nor{\nabla g_n}_{2p,R_n}^2}  \\
&-a\delta d\alpha_{t_n}^{-2}\nor{\nabla g_n}_{2,R_n}^2 -C\alpha_{t_n}^{-d/q}R_n^{-\epsilon}N_{2,R_n\alpha_{t_n}}^2(h_n).
\end{align*}
 We recall that $\nor{\nabla g_n}_{2,R_n}=\alpha_{t_n}^{1-d/2q}N_{2,R_n\alpha_{t_n}}(\tilde\nabla h_n)$. Moreover, as $\nor{h_n}_{\infty}\leq 1$, $\nor{\nabla g_n}_{2p,R_n}=\alpha_{t_n}N_{2p,R_n\alpha_{t_n}}(\tilde\nabla h_n)\leq \alpha_{t_n}N_{2,R_n\alpha_{t_n}}^{1/p}(\tilde\nabla h_n)$. We deduce:

 \begin{align*}
&\frac{a}{\alpha_{t_n}^{d/q}}  N_{2,R_n\alpha_{t_n}}^2(h_n)+\frac{1}{2}\alpha_{t_n}^{2-d/q}N_{2,R_n\alpha_{t_n}}^2(\tilde\nabla h_n)\\
\geq &\min\pare{\frac{\delta}{\delta+1},1- \frac{C}{R_n^\epsilon}}   \pare{1-CR_n^\frac{\epsilon-1}{2p}}^2\rho(a)\pare{\frac{1}{1+\gamma}  -
\frac{C}{\gamma}N_{2,R_n\alpha_{t_n}}^{2/p}(\tilde\nabla h_n)}  \\
&-a\delta d\alpha_{t_n}^{-d/2q}N_{2,R_n\alpha_{t_n}}^2(\tilde\nabla h_n) -C\alpha_{t_n}^{-d/q}R_n^{-\epsilon}N_{2,R_n\alpha_{t_n}}^2(h_n).
\end{align*}

 Let $n\rightarrow +\infty$, as $2>d/q$:
\begin{align*}
&\liminf_n \frac{a}{\alpha_{t_n}^{d/q}}  N_{2,R_n\alpha_{t_n}}^2(h_n)+\frac{1}{2}\alpha_{t_n}^{2-d/q}N_{2,R_n\alpha_{t_n}}^2(\tilde\nabla h_n)\\
=&\liminf_n\frac{a}{\alpha_{t_n}^{d/q}}  N_{2,R_n\alpha_{t_n}}^2(h_n)\pare{1+\frac{C}{aR_n^\epsilon}}   +\frac{1}{2}\alpha_{t_n}^{2-d/q}N_{2,R_n\alpha_{t_n}}^2(\tilde\nabla h_n)\pare{1+\frac{2ad\delta}{\alpha_{t_n}^{2-d/2q}}}\\
\geq & \liminf_n \min\pare{\frac{\delta}{\delta+1},1- \frac{C}{R_n^\epsilon}}   \pare{1-CR_n^\frac{\epsilon-1}{2p}}^2\rho(a)\pare{\frac{1}{1+\gamma}  -
\frac{C}{\gamma}N_{2,R_n\alpha_{t_n}}^{2/p}(\tilde\nabla h_n)}.
\end{align*}

In one hand, if $\alpha_{t_n}^{2-d/q}N_{2,R_n\alpha_{t_n}}(\tilde\nabla h_n)\rightarrow +\infty$, then the result is obvious. In the other hand, if  $\alpha_{t_n}^{2-d/q}N_{2,R_n\alpha_{t_n}}(\tilde\nabla h_n)$ is converging, then $N_{2,R_n\alpha_{t_n}}(\tilde\nabla h_n)\rightarrow 0$ because $2>d/q$, and we have:

 $$\liminf_n  \frac{a}{\alpha_{t_n}^{d/q}}  N_{2,R_n\alpha_{t_n}}^2(h_n)+\frac{1}{2}\alpha_{t_n}^{2-d/q}N_{2,R_n\alpha_{t_n}}^2(\tilde\nabla h_n) 
 \geq \min\pare{\frac{\delta}{\delta+1},1}\rho(a)\frac{1}{1+\gamma}.$$
 Then we let $\delta \rightarrow +\infty$ and $\gamma \rightarrow 0$ to have the result.\\
 $ $\\
 \textbf{Proof of proposition \ref{cste}:}\\
 We recall that $\rho(a)=\inf\acc{a\nor{h}_2^2+\frac{1}{2}\nor{\nabla h}_2^2,\nor{h}_{2p}=1}$. 
 Set $h_\beta (\cdot)=\beta^{d/2p}h(\beta \cdot)$. We remark that $\nor{h_\beta}_{2p}=1$, $\nor{h_\beta}_2=\beta^{-d/2q}\nor{h}_2$ and $\nor{\nabla h_\beta}_2=\beta^{1-d/2q}\nor{\nabla h}_2$.
 Then we minimize over $\beta$ the function: 
 $$\Phi_h(\beta)=a\beta^{-d/q}\nor{h}_2^2+\frac{1}{2}\beta^{2-d/q}\nor{\nabla h}_2^2 .$$
 Picking the optimal value $\beta^*=\sqrt\frac{2ad}{2q-d}\frac{\nor{h}_2}{\nor{\nabla h}_2}$ we have that 
 $$\rho(a)=a^{1-d/2q}\pare{\frac{2q}{2q-d}}\pare{\frac{2q-d}{2d}}^{d/2q}   \inf\acc{\nor{h}_2^{2-d/q}\nor{\nabla h}_2^{d/q},\nor{h}_{2p}=1}.$$
Then optimizing over $a>0$ the expression $\inf\acc{a-\rho(a),a>0}$ with the optimal value 
$$a^*= \frac{2q-d}{2d}\inf\acc{\nor{h}_2^{4q/d-2}\nor{\nabla h}_2^2,\nor{h}_{2p}=1}$$  we have that $$\inf\acc{a-\rho(a),a>0}   =   -\inf\acc{\frac{1}{2} \nor{h}_2^{4q/d-2}\nor{\nabla h}_2^2,\nor{h}_{2p}=1}.$$ 
 Note that the expression is invariant by the transformation $h_{\beta}(\cdot)=\beta^{d/2p}h(\beta\cdot)$, therefore we can freely add the condition $\nor{h}_{2}=1$. Then 
 $$\inf\acc{a-\rho(a),a>0}   =   -\inf\acc{\frac{1}{2}\nor{\nabla h}_2^2,\nor{h}_{2p}=\nor{h}_2=1}:=-\chi_{d,p}.$$

\textsc{Cl\'ement Laurent\\
LATP, UMR CNRS 6632\\
CMI, Universit\'e de Provence\\
39 rue Joliot-Curie, F-13453 Marseille cedex 13, France\\
laurent@cmi.univ-mrs.fr  }
 
\end{document}